\newif\ifarxiv
\arxivtrue
\documentclass[runningheads]{llncs}
\usepackage[T1]{fontenc}

\usepackage{xspace}
\usepackage{graphicx}
\usepackage{subcaption}

\usepackage{amssymb}
\usepackage{amsmath}
\usepackage{thm-restate}
\usepackage{microtype}

\spnewtheorem{obs}{Observation}{\bfseries}{\itshape}

\title{Weakly and Strongly Fan-Planar Graphs}
\author{Otfried Cheong\inst{1}\orcidID{0000-0003-4467-7075}\and Henry F\"orster\inst{2}\orcidID{0000-0002-1441-4189}
\and Julia Katheder\inst{2}\orcidID{0000-0002-7545-0730}
Maximilian Pfister\inst{2}\orcidID{0000-0002-7203-0669 }\and 
Lena Schlipf\inst{2}\orcidID{0000-0001-7043-1867} }
\authorrunning{Cheong et al.}
\institute{SCALGO, \email{otfried@scalgo.com}
\and 
Wilhelm-Schickard-Institut f{\"u}r Informatik, Universit{\"a}t T{\"u}bingen, 
\email{\{henry.foerster, julia.katheder, maximilian.pfister, lena.schlipf\}@uni-tuebingen.de}}

\graphicspath{{figures/}}

\usepackage{hyperref}

\newcommand{\reffani}{(I)\xspace}
\newcommand{\reffanii}{(II)\xspace}
\newcommand{\reffaniii}{(III)\xspace}

\begin{document}
\ifarxiv\else
\pagestyle{empty}
\fi

\maketitle

\begin{abstract}
We study two notions of fan-planarity introduced by (Cheong et al., GD22), called weak and strong fan-planarity,
which separate two non-equivalent definitions of fan-planarity in the literature. 
We prove~that not every weakly fan-planar graph is strongly fan-planar, 
while the upper bound on the edge density is the same for both families.
  
\end{abstract}

\section{Introduction}

Crossings in graph drawings are known to heavily impede readability~\cite{DBLP:conf/gd/Purchase97,DBLP:journals/ivs/WarePCM02}.  Unfortunately, however, minimizing the number of crossings is NP-complete~\cite{doi:10.1137/0604033}, while many real-world networks turn out to be non-planar. Fortunately, readable drawings of non-planar graphs can be obtained by limiting the \emph{topology}~\cite{DBLP:conf/gd/Mutzel96} or the \emph{geometry} of crossings~\cite{DBLP:conf/apvis/Huang07,DBLP:journals/vlc/HuangEH14}. Based on these experimental findings, the research direction of \emph{graph drawing beyond planarity}  has emerged. This line of research is dedicated to the study of so-called \emph{beyond planar} graph classes that are defined by forbidden edge-crossing patterns. More precisely, a graph belonging to such a class admits a drawing in which the forbidden pattern is absent. Important beyond planar graph classes are $k$-planar graphs, where the forbidden pattern is $k+1$ crossings on the same edge, $k$-quasiplanar graphs, where $k$ mutually crossing edges are prohibited, and RAC-graphs, where edges are not allowed to cross at non-right angles. We refer the interested reader to the survey by Didimo et al.~\cite{10.1145/3301281} and a recent book~\cite{Hong-Tokuyama} on beyond planarity.

In this paper, we study \emph{fan-planar graphs}, which admit \emph{fan-planar drawings}. In these drawings, each edge $e$ may only be crossed by a \emph{fan} of edges, that is a bundle of edges sharing a common endpoint, called \emph{anchor} of $e$, that all cross $e$ from the same side. Kaufmann and Ueckerdt introduced this graph class in 2014~\cite{DBLP:journals/corr/KaufmannU14} and described the aforementioned requirement with two forbidden patterns,  Pattern~\reffani and~\reffanii in Fig.~\ref{fig:fanplanar}: the first forbids two edges crossing $e$ to be non-adjacent whereas the second forbids crossings of~$e$ by adjacent edges with the common endpoint on different sides of~$e$. Since their introduction, fan-planar graphs have received a lot of attention in the scientific community; see \cite{Bekos2020} for an overview.

\begin{figure}[h]
  \centerline{\includegraphics{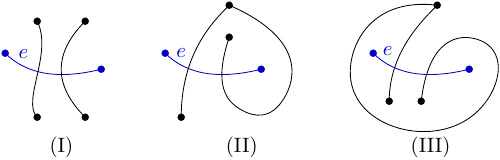}}
  \caption{Forbidden configurations.}
  \label{fig:fanplanar}
\end{figure}

Recently, Klemz et al.~\cite{DBLP:journals/jgaa/KlemzKRS23} pointed out a missing case in the proof 
of the edge density upper-bound in the preprint 
that introduced fan-planar graphs~\cite{DBLP:journals/corr/KaufmannU14}.
This case was consequently fixed in the journal
version~\cite{DBLP:journals/combinatorics/0001U22} by the original authors. 
However, in the process, they introduced a third forbidden pattern, 
namely Pattern~\reffaniii of Fig.~\ref{fig:fanplanar},
which is quite reminiscent of the previously defined Pattern~\reffanii. 
Namely, in both patterns, the drawing restricted to the edge~$e$ and the two edges crossing~$e$
has two connected regions, called \emph{cells}, where one is bounded and the other one is unbounded. 
The difference between the two configurations is that in Pattern~\reffanii, 
one endpoint of~$e$ lies in the bounded cell, while in
Pattern~\reffaniii, both endpoints are contained there. 
In the new definition of \emph{fan-planarity}, both endpoints of~$e$ must lie in the unbounded cell. 

The new forbidden Pattern~\reffaniii poses a problem for the existing literature on fan-planarity. 
Namely, while some previous results still apply when forbidding Pattern~\reffaniii, other existing results, 
e.g., the bound of~$4n-12$ on the number of edges of $n$-vertex bipartite ``\emph{fan-planar}'' 
graphs~\cite{DBLP:conf/isaac/AngeliniB0PU18}, 
build on the lemmas of the original paper~\cite{DBLP:journals/corr/KaufmannU14} 
and may be affected by the recent changes to the definition.

Cheong et al.~\cite{DBLP:conf/gd/CheongPS22} introduced a clear distinction of the two models 
with the notions of \emph{weak} and \emph{strong fan-planarity}. 
Namely, \emph{weak fan-planarity} allows  Pattern~\reffaniii whereas \emph{strong
fan-planarity} forbids it. 
The original definition of fan-planarity~\cite{DBLP:journals/corr/KaufmannU14} coincides 
with \emph{weak fan-planarity} whereas \emph{strong fan-planarity} matches 
the definition of the new journal version~\cite{DBLP:journals/combinatorics/0001U22}.
Graphs admitting such drawings are called \emph{weakly fan-planar} and \emph{strongly
fan-planar}, respectively. 

The family of graphs that admits drawings where only
Pattern~\reffani  is forbidden, called \emph{adjacency-crossing} graphs, has also been studied by    Brandenburg~\cite{DBLP:journals/tcs/Brandenburg20}. He showed that there are adjacency-crossing graphs that are \emph{not}
weakly fan-planar, so weakly fan-planar graphs form a proper subset of
the adjacency-crossing graphs. Moreover, he shows that for any
$n$-vertex adjacency-crossing graph with $m$~edges, one can construct
a weakly fan-planar graph with~$n$ vertices that also has $m$~edges.
Brandenburg concluded from this that an $n$-vertex adjacency-crossing
graph has at most~$5n-10$ edges, since that was the bound claimed
by~\cite{DBLP:journals/corr/KaufmannU14} for ``fan-planar'' graphs.  
Since that bound holds only
under strong fan-planarity, this conclusion contains a gap, which the
present paper fills.

\paragraph{Our contribution.}

First, we prove that the family of strongly fan-planar graphs is a proper subset of the weakly fan-planar graphs.
Together with Brandenburg's result, this implies that the two
inclusions of strongly fan-planar graphs inside weakly fan-planar graphs and
weakly fan-planar graphs inside
adjacency-crossing graphs are both proper.
We then continue to show that the known upper bound on the
edge-density of strongly fan-planar graphs (namely~$5n-10$ for an
$n$-vertex graph) carries over to weakly fan-planar graphs.  This
implies that also Brandenburg's
bound~\cite{DBLP:journals/tcs/Brandenburg20} is in fact correct.
We also prove that the known upper bound of~$4n-12$ on the number of
edges of an  $n$-vertex  \emph{bipartite} strongly fan-planar graph
carries over to bipartite weakly fan-planar graphs.

\section{Not every weakly fan-planar graph is strongly fan-planar}

In this section, we will establish that strongly fan-planar graphs
form a proper subset of weakly fan-planar graphs by constructing a
graph~$G$ with a weakly fan-planar drawing~$\Gamma$,
where Pattern~\reffaniii cannot be avoided in~$\Gamma$.

In order to guarantee the existence of at least one Pattern~\reffaniii 
in any valid weakly fan-planar drawing of~$G$, we will use the following key idea.
We start with a planar graph with a unique embedding. 
We will then make every edge of this planar graph ``uncrossable'' 
by replacing it with a suitable gadget introduced by Binucci et al.~\cite{BINUCCI201576}. 
Afterwards, we insert into every face of the planar graph a small gadget graph 
(shown in Fig.~\ref{fig:forced_heart}),
which can only be drawn with a quadrangular outer face if we allow \reffaniii. Note that this gadget graph itself is in fact strongly fan-planar, as shown in Fig.~\ref{fig:h-alternative}, and, hence, does not serve itself as an example of a weakly but not strongly fan-planar graph.
In order to achieve our goal, we will leverage the following lemma:
\begin{lemma}[Binucci et al.~\cite{BINUCCI201576}] 
    \label{lem:uncrossed_spine}
    Let $\mathcal{P}$ be the planarization\footnote{In a \emph{planarization} $\cal P$ of a non-planar drawing of a graph $G$, each crossing is replaced with a dummy-vertex that subdivides both edges involved in the crossing. We call an edge of $\cal P$ that is not incident to any dummy-vertex a \emph{real edge of $G$}.} of any weakly fan-planar drawing of~$K_7$. 
    Then, between any pair of vertices of~$K_7$, 
    there exists a path in~$\mathcal{P}$ that contains no real edge of the~$K_7$.
\end{lemma}

Moreover, we will use the following definition throughout the paper.
\begin{definition}
    Let $\Gamma$ be a weakly fan-planar drawing of a graph~$G$. $\Gamma$~is said to be \emph{minimal},
    if, among all weakly fan-planar drawings of $G$, 
    it contains the smallest possible number of triples of edges that form Pattern~\reffaniii.
\end{definition}

\begin{theorem}
  \label{thm:strong-vs-weak}
  There exists a weakly fan-planar graph that does not admit a
  strongly fan-planar drawing.
\end{theorem}
\begin{proof}
Let $G_0$ be a $3$-connected planar quadrangulation; e.g, one that is obtained by the construction in~\cite{DBLP:journals/dm/BrinkmannGGMTW05}. 
Note that by construction, $G_0$ is bipartite and has a unique embedding into~$\mathbb{R}^2$ up to the choice of the outer face and a mirroring~\cite{whitney1992congruent}. 
Next, we insert a copy of our gadget graph~$H$ shown in Fig.~\ref{fig:forced_heart} 
into every face~$f$ of~$G_0$ by identifying the outer cycle of~$H$ with the facial cycle of~$f$. 
Denote by~$G_1$ this supergraph of~$G_0$. 
We use the color scheme of Fig.~\ref{fig:forced_heart} to color all edges of $G_1$ -- in 
particular, the edges of $G_0$ form a subset of the red edges of~$G_1$. 
In the next step, we substitute every \emph{red} edge of~$G_1$ by a~$K_7$ 
and denote the resulting graph by~$G$. 
We claim that~$G$ is weakly fan-planar, but not strongly fan-planar.

For the first statement, 
observe that~$K_7$ admits a weakly fan-planar drawing, see Fig.~\ref{fig:k7_fan},
and that our gadget graph~$H$ has a weakly fan-planar drawing shown in Fig.~\ref{fig:forced_heart}. Combining both, we obtain a weakly fan-planar drawing of~$G$. 

Consider now the second statement. 
Let $\Gamma$ be a weakly fan-planar drawing of~$G$ that is minimal 
and that has the smallest number of crossings among all minimal weakly fan-planar drawings of~$G$.
We will prove that $\Gamma$ contains at least one Pattern~\reffaniii, 
which implies by our choice of $\Gamma$ that \emph{every} weakly fan-planar drawing of~$G$ 
requires at least one Pattern~\reffaniii.

Consider a red edge $ab \in G_1$ and denote by $a = v_1,\dots,v_7 = b$ the vertices of the~$K_7$ 
which substitute $ab$ in~$G$. 
By Lemma~\ref{lem:uncrossed_spine}, in the drawing~$\Gamma$ of this~$K_7$,
there exists a sequence~$S$ of crossed edges from~$a$ to~$b$; see Fig.~\ref{fig:k7_fan}.
By~\reffani, no edge which is not incident to one of $v_1,\dots,v_7$ can intersect~$S$. 
By construction, the only edges incident to vertices $v_2,\dots,v_6$ are edges of the $K_7$.
Hence, the only edges that can potentially cross~$S$ and interact with the remainder of~$G$ 
are incident to either $a = v_1$ or $b = v_7$. 
Suppose for a contradiction that there exists an edge incident to~$a$ or~$b$ 
that crosses~$S$ in~$\Gamma$ such that its other endpoint is not one of the vertices of the~$K_7$. 
But then we can easily reroute the edge such that its crossing with~$S$ is avoided, 
see Fig.~\ref{fig:rerouting_k7}, a contradiction to our choice of~$\Gamma$.
\begin{figure}[t]
	\begin{subfigure}[b]{.48\textwidth}
		\centerline{\includegraphics[scale=1,page=1]{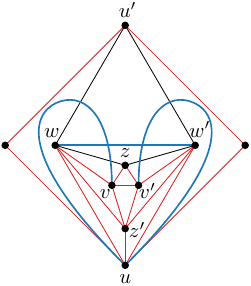}}
		\subcaption{}
		\label{fig:forced_heart}
	\end{subfigure}
	\hfill
        \begin{subfigure}[b]{.48\textwidth}
		\centerline{\includegraphics[scale=1,page=4]{forced_heart}}
		\subcaption{}
		\label{fig:h-alternative}
	\end{subfigure}
	\hfill
	\begin{subfigure}[b]{.48\textwidth}
		\centerline{\includegraphics[scale=1,page=3]{forced_heart}}
		\subcaption{}
		\label{fig:k7_fan}
	\end{subfigure}
	\hfill
    \begin{subfigure}[b]{.48\textwidth}
	   \centerline{\includegraphics[scale=1,page=2]{forced_heart}}
	   \subcaption{}
	   \label{fig:rerouting_k7}
    \end{subfigure}
	\caption{(a) Gadget graph $H$. 
(b) Gadget graph $H$ with $(z,v,v')$ chosen as the outer face. Note that there is no pattern \reffaniii.
(c) A planarization of a fan-planar drawing of $K_7$, where the bold edges form a path from $a$ to $b$ that contains no \emph{uncrossed edge} of the $K_7$.
 (d) An edge incident to $a$ that crosses the~$K_7$ to avoid this crossing.}
	\label{fig:mix}
\end{figure}

We interpret $\Gamma$ as a drawing $\Gamma'$ of~$G_1$, 
where the red edges are \emph{uncrossed}.
Since~$G_0$ has a unique planar embedding into~$\mathbb{R}^2$ up to the choice of the outer face and a mirroring and since $G_0$ consists only of  the red edges, in $\Gamma'$, $G_0$ is drawn as a planar graph, all faces of which are quadrilaterals.
Since the red edges are uncrossed, each quadrilateral face must contain a copy of our gadget~$H$.
Indeed, since each vertex of~$H$ is connected by a path to both~$u$ and~$u'$, 
it must lie in a face that contains both~$u$ and~$u'$.
But by $3$-connectivity of~$G_0$, this face is unique.

Let $f$ be a bounded quadrilateral face of~$G_0$. As argued above, $f$ contains a copy of~$H$ in its interior in~$\Gamma'$, i.e. vertices $v,v',w,w',z$ and $z'$, refer to Fig.~\ref{fig:forced_heart}, 
lie in the interior of~$f$.
Consider the subgraph~$H'$ of this~$H$ consisting of its red edges.
Since the red edges are uncrossed, $H'$ is drawn without crossings in~$\Gamma$.
Since the black edges do not cross~$H'$, a small case analysis shows 
that the vertices~$u'$, $w$, $z$, and~$w'$ must lie in the same face of~$H'$,
and in fact the embedding of Fig.~\ref{fig:forced_heart} is unique up to symmetry
with respect to the vertical axis.
Thus, the blue edges $(u,v)$, $(u,v')$ and $(w,w')$ can
only be drawn in the indicated way,
but that means that the blue edges form~\reffaniii,
which concludes the proof.
\qed\end{proof}

\section{Density of weakly fan-planar graphs}
\label{sec:density-weak-fp}
In this section, we show that the density results for strongly fan-planar graphs also transfer to the weakly fan-planar setting. 
Let us call a triple $e, e_\ell, e_r$ of edges in a weakly fan-planar drawing 
a \emph{heart} if~$e_\ell$ and~$e_r$ share an endpoint~$u$, 
both cross~$e$ such that they form Pattern~\reffaniii,
and the part of~$e$ between the crossings with~$e_\ell$ and~$e_r$
is not crossed by any edge of the graph, see Fig.\ref{fig:heart}.
In the remainder of the paper, we call the intersection points of $e$ and $e_\ell$ ($e_r$, resp.) $x_\ell$ ($x_r$, resp.).

\begin{figure}[t]
\centering
\includegraphics[page=10]{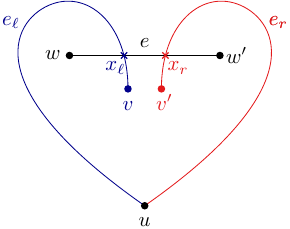}
\caption{A heart.}
\label{fig:heart}
\end{figure}

\begin{lemma}\label{lem:heart}
Let~$\Gamma$ be a weakly fan-planar drawing that is not strongly fan-planar. Then~$\Gamma$ contains a heart $\mathcal{H}$.
\end{lemma}

\begin{proof}
By assumption, $\Gamma$ contains three edges~$e, e_\ell, e_r$ that form Pattern~\reffaniii, where~$e_\ell$ and~$e_r$ share endpoint~$u$ and cross~$e$. Let~$E'$ be the set of edges that cross~$e$.
By~\reffani and~\reffanii, any edge~$e' \in E'$ must be incident to~$u$, and by~\reffanii it must cross~$e$ from the same side as~$e_\ell$ and~$e_r$. The edges of~$E'$ cannot cross each other since they share an endpoint, and each edge~$e' \in E'$ forms Pattern~\reffaniii either with~$e$ and~$e_r$, or with~$e$ and~$e_\ell$.
Let $E_\ell \subset E'$ be the set of edges of the first kind, $E_r = E' \setminus E_\ell$ the second kind.
If we order~$E'$ by their crossing point with~$e$ along~$e$, 
then we first encounter all elements of~$E_\ell$, then all elements of~$E_r$. 
The last element of~$E_\ell$ and the first element of~$E_r$ form a heart with~$e$.
\qed\end{proof}
We will call the sets $E_\ell$ and $E_r$ as defined in the previous proof the 
\emph{left valve} and the \emph{right valve} of the heart~$\mathcal{H}=e, e_\ell, e_r$, respectively.  
We denote by $H$ the edge set containing both valves of $\mathcal{H}$ and the edge $e$, 
namely  $H = E_\ell \cup E_r \cup {e}$.
In the following, we will define an edge-rerouting operation 
that will later allow us to reduce the number of hearts in a weakly fan-planar drawing under certain conditions.

\paragraph{Flipping the valve of a heart.}
Consider a heart $\mathcal{H}$ formed by the edges $e,e_\ell,e_r$ in a weakly fan-planar drawing $\Gamma$; refer to Fig.~\ref{fig:rerouting-pre} for a visualization. 
In the following, we will define an operation that we call \emph{flipping} a valve of $\mathcal{H}$ resulting in the drawing $\Gamma'$. 
We describe the flip of $E_\ell$, as the other case is symmetric. 
The general idea is to redraw the edges in $E_\ell$ ``close'' to the ones in the other valve $E_r$, 
in particular, mainly following the curve of $e_r$.

\begin{figure}[t]
\centering
\begin{subfigure}[b]{.48\textwidth}
\centering
\includegraphics[scale=1,page=1]{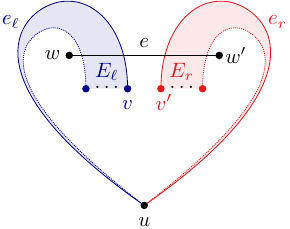}
\subcaption{}
\label{fig:rerouting-pre}
\end{subfigure}
\hfil
\begin{subfigure}[b]{.48\textwidth}
\centering
\includegraphics[scale=1,page=2]{figures/double-heart}
\subcaption{}
\label{fig:rerouting-result}
\end{subfigure}
\caption{(a) Illustration of the setting previous to the flip-operation. (b) Transformation from $\Gamma$ to $\Gamma'$ by flipping $E_\ell$.}
\label{fig:rerouting}
\end{figure}

Let $e^\ell_1,e^\ell_2,\dots e^\ell_k$ be the edges of $E_\ell$ in the order that they intersect edge $e$ in $\Gamma$ starting at $w$, i.e. $e^\ell_k=e_\ell$. We will draw the curve $\gamma^i$ of $e^\ell_i$ in three parts, denoted as $\gamma^i_1, \gamma^i_2, \gamma^i_3$. We consider the edges in reverse order and start with $e^\ell_k = e_\ell$. The first part $\gamma^k_1$ of the curve of $e^\ell_k$ in $\Gamma'$ follows the curve of $e_r$ slightly outside until $x_r$, then $\gamma^k_2$ follows $e$ until $x_\ell$, where the curve intersects $e$ and afterwards it inherits its original curve in $\Gamma$ as the last part $\gamma^k_3$. So, assume that we have drawn $e^\ell_i$ with $2<i\leq k$. The curve of $e^\ell_{i-1}$ follows the curve of $e^\ell_i$ (slightly outside) until $x_r$, where it follows $e$ until the intersection point of $e^l_{i-1}$ with $e$ in $\Gamma$. Here, the curve intersects $e$ and then again inherits its original curve in $\Gamma$ until it reaches its endpoint different from $u$.
After this operation, the edges of $E_\ell \cup E_r$ do not cross by simplicity and $e$ does not form Pattern \reffani to \reffaniii with $E_\ell \cup E_r$, thus we can make the following observation, illustrated in Fig.~\ref{fig:rerouting-result}.
\begin{obs} \label{obs:flip}
Let $\Gamma'$ be the drawing obtained from flipping a valve of a heart $\mathcal{H}$. Then, $\Gamma'[H]$ is a strongly fan-planar drawing.
\end{obs}

\noindent Moreover, in the entire drawing $\Gamma'$ (not limited to $H$) resulting from a flip, 
new crossings can arise only in a restricted part of the flipped edges.

\begin{lemma} \label{lem:crossing-first-part}
Let $\mathcal{H}$  be a heart formed by edges $e, e_\ell, e_r$ in $\Gamma$ and $E_\ell = \{e^\ell_1,e^\ell_2, \dots e^\ell_k\}$ be the edges of a flipped valve of $\mathcal{H}$ in $\Gamma'$. Then any crossing introduced by the flipping operation occurs on the partial curves $\gamma^1_1, \gamma^2_1, \dots \gamma^k_1$.
\end{lemma}
\begin{proof}
Let us consider a flipped edge $e^\ell_i \in E_\ell$ and its curve $\gamma^i$ in $\Gamma'$. Clearly, any additional crossing that we introduce may only occur on the first part $\gamma^i_1$ or second part $\gamma^i_2$ of $\gamma^i$, as the last part $\gamma^i_3$ is inherited from $\Gamma$. By construction, $\gamma^i_2$ is crossing free, as the segment of $e$ between $x_\ell$ and $x_r$ in $\Gamma$ is crossing-free since $e_\ell,e_r$ and $e$ form a heart and by considering the edges in the reverse order that they intersect $e$ (starting at $w$), they do not intersect each other.
\qed\end{proof}

\noindent For our later proofs of Theorems~\ref{thm:density}~and~\ref{thm:density-bipartite} on the edge density of  (bipartite) weakly fan-planar drawings, we need to show that a certain configuration, as described in the following lemma, cannot occur in a \emph{minimal} weakly fan-planar drawing.

\begin{lemma} \label{lem:no-two-edges}
Let $e, e_\ell, e_r$ be a heart in a minimal weakly fan-planar drawing~$\Gamma$. Then there is no edge~$e' \neq e$ in~$\Gamma$ that crosses both~$e_\ell$ and~$e_r$.
\end{lemma}
\begin{proof}
Assume for a contradiction that there exists an edge~$e' \neq e$ that intersects both~$e_\ell$ and~$e_r$. By~\reffani, $e$ and~$e'$ share an endpoint, say~$w$, see Fig.~\ref{fig:e-single-heart}. This implies by~\reffani and~\reffanii that \emph{every} edge which crosses~$e_\ell$ or~$e_r$ is incident to~$w$. W.l.o.g. assume that $x_\ell$ is encountered before $x_r$ when traversing $e$ starting at $w$.
First, notice that $e$, $e'$ and $e_r$ can in turn form \reffaniii themselves, refer to Fig.~\ref{fig:e-double-heart}. 

\begin{figure}[t]
\centering
\begin{subfigure}[b]{.48\textwidth}
\centering
\includegraphics[scale=1,page=3]{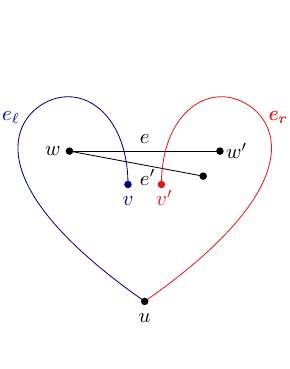}
\subcaption{}
\label{fig:e-single-heart}
\end{subfigure}
\hfil
\begin{subfigure}[b]{.48\textwidth}
\centering
\includegraphics[scale=1,page=4]{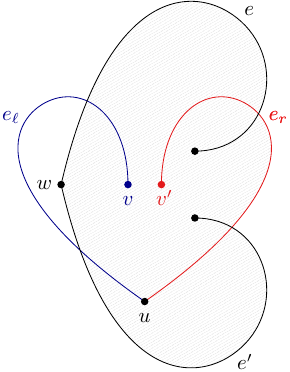}
\subcaption{}
\label{fig:e-double-heart}
\end{subfigure}
\caption{(a) A single heart $\mathcal{H}=e,e_\ell,e_r$. (b) A double heart formed by $\mathcal{H}=e,e_\ell,e_r$ and $\mathcal{H}'=e_r,e,e'$, the area bounded by $\mathcal{H}'$ is indicated in gray. }
\end{figure}

Based on this observation, we will define four sets of edges that will be helpful in the remainder. In particular, we construct these sets based on their drawing in $\Gamma$---while some of the edges may be redrawn at a later time, they will always belong to their corresponding sets. 
Let $E_\ell$ and $E_r$ be the sets of edges that correspond to the left valve and the right valve of the heart $\mathcal{H}$ induced by $e,e_\ell,e_r$---in particular, $e_\ell \in E_\ell$ and $e_r \in E_r$. 
Further, let $E_t$ be the set of edges that cross both $e_\ell$ and $e_r$ and do not form \reffaniii with $e_r$ and $e$ when traversing $e_\ell$  starting at $v$. Complementary, let $E_b$ be the set of edges that also cross $e_\ell$ and $e_r$ and are not contained in $E_t$. Further, let $e$ be the first edge in $E_t$ that is encountered when traversing $e_\ell$ ($e_r$) starting at $u$.
In the following, we will distinguish between two cases. First, if $E_b = \emptyset$, we call $\mathcal{H}$ a \emph{single heart}. 
Otherwise, that is when~$E_b \neq \emptyset$, 
we assume that~$e'$ is the last edge in~$E_b$ that is encountered when traversing~$e_\ell$ starting at~$u$
(and therefore also when traversing~$e_r$). 
We say that~$\mathcal{H}$ forms a \emph{double heart} with $\mathcal{H}'$ defined by the edges $e_r,e, e'$; see Fig.~\ref{fig:hearts-cases}.

\begin{figure}[t]
\centering
\begin{subfigure}[b]{.48\textwidth}
\centering
\includegraphics[scale=1,page=5]{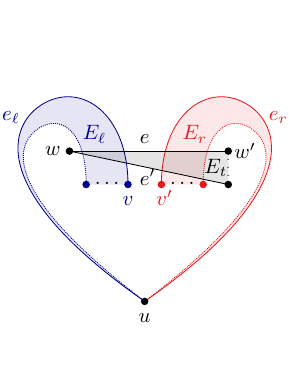}
\subcaption{}
\label{fig:case-single-heart}
\end{subfigure}
\hfil
\begin{subfigure}[b]{.48\textwidth}
\centering
\includegraphics[scale=1,page=6]{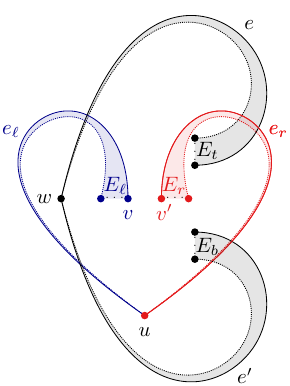}
\subcaption{}
\label{fig:case-double-heart}
\end{subfigure}
\caption{Illustration of edge sets used in (a) a single heart and (b) a double heart.}
\label{fig:hearts-cases}
\end{figure}
The main idea is to apply the flip operation to reduce the number of hearts in the resulting (weakly fan-planar, as to be shown) drawing $\Gamma'$ to obtain a contradiction to the minimality of $\Gamma$. For a single heart, one flip will suffice while two flips will be necessary for a double heart. In both cases, we start by flipping the valve $E_\ell$ as previously defined to obtain the drawing $\Gamma'$. We first assert:

\paragraph{$\Gamma'$ is weakly fan-planar.}
By Observation~\ref{obs:flip}, any new edge triple forming Pattern \reffani or \reffanii in $\Gamma'$ must involve at least one edge in the flipped valve $E_\ell$ and at least one edge belonging to neither $E_\ell$ nor $E_r$. Consider a flipped edge $e^\ell_i \in E_\ell$. By Lemma~\ref{lem:crossing-first-part}, any new crossing on the curve $\gamma^i$ of $e^\ell_i$ may only occur on its first part $\gamma^i_1$ in $\Gamma'$. Any edge that intersects $\gamma^i_1$ in $\Gamma'$ also intersects $e_r$ and is therefore incident to $w$ by Pattern \reffani.
To show that we do not introduce new edge-triples forming Pattern \reffani, it remains to show that all edges in $E_\ell$ have vertex $w$ as their anchor, even if they only cross $e$ and no other edge in $E_t$. 

We consider two cases. If $\mathcal{H}$ forms a double heart with $\mathcal{H}'$, consider any edge $e_b \in E_b$. Since $e \in E_t$, we observe that each edge $e_\ell \in  E_\ell$ is crossed both by $e$ and $e_b$; see Fig.~\ref{fig:anchor-doubleheart}. Thus, its anchor is $w$. Otherwise, $\mathcal{H}$ is a single heart. Consider the region $\mathcal{R}$ defined by $w$, $x_\ell$ and the intersection of $e_\ell$ and another edge $e_t \in E_t$ in the original drawing $\Gamma$, see Fig.~\ref{fig:region-R}
for an illustration of this case.

\begin{figure}[t]
\centering
\begin{subfigure}{0.4\textwidth}
    \centering
\includegraphics[page=11]{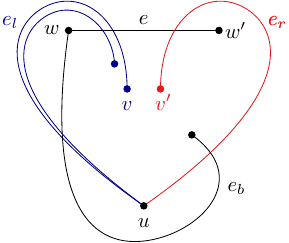}
\caption{}
\label{fig:anchor-doubleheart}
\end{subfigure}
\hfil
\begin{subfigure}{0.4\textwidth}
    \centering
\includegraphics[page=9]{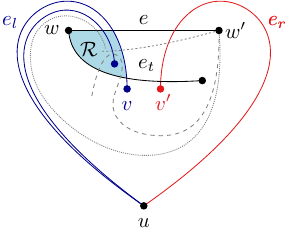}
\caption{}
\label{fig:region-R}
\end{subfigure}
\caption{(a) Every edge in $E_\ell$ has anchor $w$ if $\mathcal{H}$ is part of a double heart. (b) Any edge of $E_\ell$ that enters region $\mathcal{R}$ has $w$ as its anchor, if $\mathcal{H}$ forms a single heart.}
\label{fig:anchors-in-lemma4}
\end{figure}
	
By definition, every edge in $E_\ell$ besides $e_\ell$ enters $\mathcal{R}$ over $e$. If such an edge is leaving $\mathcal{R}$, then it has to also cross $e_t$ by simplicity, but then its anchor is $w$ and by \reffani it can only be crossed by edges that are incident to $w$. Suppose now an edge of $E_\ell$ ends in $\mathcal{R}$, but its anchor is not $w$. Since the edge intersects $e$, its anchor is therefore $w'$. But no edge incident to $w'$ can enter $\mathcal{R}$, since it would either cross $e_\ell$, whose anchor is $w$ and hence it would coincide with $e=(w,w')$, it cannot cross $e$ by simplicity and if it crosses $e_t$, then it has to be incident to the anchor of $e_t$, which is $u$, but then the curve has to intersect the boundary of $\mathcal{R}$ twice which is impossible.

It remains to show that we do not introduce Pattern \reffanii by the flip-operation. For the sake of contradiction, assume that there is a new edge triple $T$ forming Pattern \reffanii after the flip. Note that $T$ involves at least one flipped edge $e_i^\ell \in E_\ell$. We will first establish that the anchor of $T$ is either $w$ or $u$. Suppose that $e_i^\ell$ is not incident to the anchor of $T$, then the anchor of $T$ is $w$ (since all edges in $E_\ell$ have $w$ as their anchor). Otherwise, the edge crossing $e_i^\ell$ in $T$ also crosses $e_r$, hence $u$ is the anchor of $T$. In the case that $w$ is the anchor of $T$, observe that all edges in $E_\ell \cup E_r$ are crossed from the same side by edges incident to $w$ when directed from $u$ to their other endvertex. Hence, $T$ cannot form Pattern \reffanii. If $u$ is the anchor of $T$, it would require a non-empty region between $\gamma^i_1$ and another partial curve $\gamma^j_1$ with $j \neq i$ to form Pattern \reffanii, which is a contradiction to our construction.

Thus, we have established that $\Gamma'$ is indeed weakly fan-planar and we can now consider the number of Patterns \reffaniii. We distinguish two cases. 

\paragraph{Case 1: $\mathcal{H}$ forms a single heart.}

By Observation~\ref{obs:flip}, 
the drawing $\Gamma'[H]$ does not contain any edge triple forming Pattern \reffaniii. Hence, any new Pattern~\reffaniii in $\Gamma'$ involves either one or two flipped edges of $E_\ell$. In the latter case, since edges in $E_\ell$ do not cross by simplicity, this would require a non-empty region between the first partial curves of two edges in $E_\ell$, which is a contradiction to our construction. 

Suppose now that one edge of $E_\ell$ and two edges $e_1, e_2$ incident to $w$ form a new Pattern \reffaniii. Since the edges in $E_\ell$ follow the curve of $e_r$ after the flip, $e_r$ is necessarily crossed by one of the two edges incident to $w$, say $e_1$, in order to contain $u$ in a closed region. But then $e_1$ would be already present in $\Gamma$ and belong to the set $E_b$, as $e, e_\ell, e_r$ form Pattern \reffaniii, a contradiction to $E_b = \emptyset$. 

We conclude that no new Pattern \reffaniii is introduced by Observation~\ref{obs:flip} while all triples in $H$ forming \reffaniii are eliminated, i.e.,  the overall number  is reduced.

\paragraph{Case 2: $\mathcal{H}$ and $\mathcal{H}'$ form a double heart.} First, observe that all edges in the valves $E_\ell \cup E_r$ of $\mathcal{H}$ are anchored by $w$ and all edges of the valves $E_t \cup E_b$ of $\mathcal{H}'$ are anchored by $u$. In particular, every edge \emph{forming} $\mathcal{H}$ and $\mathcal{H}'$ is contained in one of the four edge sets $E_\ell,E_r,E_t,E_b$. Suppose now that there is a third heart $\mathcal{H}''$ distinct from $\mathcal{H}$ that forms a double heart with $\mathcal{H}'$. Hence, there are crossings between the valves of $\mathcal{H}'$ and $\mathcal{H}''$. It follows, that the edges in the valves of $\mathcal{H}''$ are necessarily incident to $u$ by \reffani, but then all crossings between edges of $H'$ and $H''$ occur from the same side as between edges of $H$ and $H'$, as otherwise Pattern \reffanii would be present in $\Gamma$. It follows that the hearts $\mathcal{H}$ and $\mathcal{H}''$ are identical. Thus, we can consider $\mathcal{H}$ and $\mathcal{H}'$ symmetrically, since in a double heart one always has the other as its counterpart. In this case, the flip of $E_\ell$ can indeed form new Patterns~\reffaniii. However, as discussed in the previous case, this can only occur when a flipped edge $e_\ell \in E_\ell$ forms a heart with an edge $e_t \in E_t$ and $e_b \in E_b$. Thus, all new Patterns~\reffaniii are contained in the single heart $\mathcal{H}'$ where $E_r$ takes the role of $E_t$ in the discussion of the previous case; see Fig.~\ref{fig:double-heart-pre}. We can now proceed as in the previous case by flipping $E_b$ which eliminates all triples forming Pattern~\reffaniii as $E_\ell=\emptyset$; see Fig.~\ref{fig:double-heart-result}. 
\qed\end{proof}

\begin{figure}[t]
\centering
\begin{subfigure}[b]{.48\textwidth}
\centering
\includegraphics[scale=1,page=7]{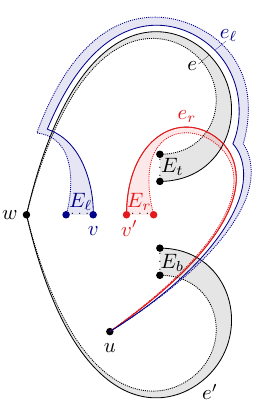}
\subcaption{}
\label{fig:double-heart-pre}
\end{subfigure}
\hfil
\begin{subfigure}[b]{.48\textwidth}
\centering
\includegraphics[scale=1,page=8]{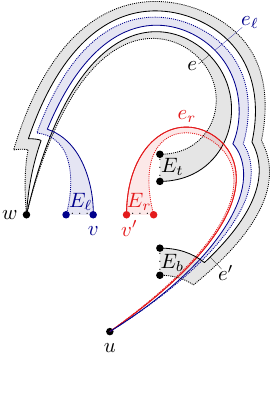}
\subcaption{}
\label{fig:double-heart-result}
\end{subfigure}
\caption{Illustration of second flip operation in the double heart case in Lemma~\ref{lem:no-two-edges}.}
\label{fig:double-heart-rerouting}
\end{figure}

So far, we investigated minimal weakly fan-planar drawings and their properties. 
One last ingredient is needed to prove Theorem~\ref{thm:density}. 
Given a graph~$G$ with a minimal weakly fan-planar drawing, 
we will sometimes reduce the number of Patterns~\reffaniii by modifying~$G$ 
into a different, weakly fan-planar graph~$G'$ with the same number of vertices 
and with the same number of edges.
The following lemma states when this modification is possible; see Appendix~\ref{app:5} for the proof:  

\begin{restatable}{lemma}{five}\label{lm:5}
Let $e_\ell,e_r$ and $e=(w,w')$ be a heart $\mathcal{H}$ in a minimal weakly fan-planar drawing $\Gamma$ of a graph $G=(V,E)$, let $\mathcal{L}$ be the closed curve that consists of the partial curves of the edges $e,e_\ell,e_r$  up to the crossing points~$x_\ell$ and~$x_r$ and let
$p_1,p_2$ be two common neighbors of $w$ and $w'$ outside of $\mathcal{L}$. Then, there is a graph $G'=(V', E')$ with $|V'|=|V|$ and $|E'|=|E|$ that admits a fan-planar drawing~$\Gamma'$ with fewer edge triples forming Pattern~\reffaniii than~$\Gamma$ contains.
\end{restatable} 
We are now ready to prove our main theorem:

\begin{theorem} \label{thm:density}
A weakly fan-planar graph~$G$ with~$n$ vertices has at most $5n-10$ edges.
\end{theorem}

\begin{proof}
    We proceed by induction on the number of edge triples forming Pattern~\reffaniii  in a \emph{minimal} weakly fan-planar drawing of~$G$. In the base case, this number is zero, so the drawing is strongly fan-planar. Then~$G$ is strongly fan-planar, and has at most $5n - 10$ edges by~\cite{DBLP:journals/combinatorics/0001U22}.
    For the induction step, consider a graph~$G$ and let $\Gamma$ be a minimal weakly fan-planar drawing of~$G$.
	
    By Lemma~\ref{lem:heart}, $\Gamma$ contains a heart~$e=(w,w'), e_\ell, e_r$, see Fig.~\ref{fig:theorem2:g}. If  $w$ and $w'$ have at least two common neighbors outside of $\cal L$, we obtain a graph $G'$ with the same number of vertices and edges with fewer edge triples forming \reffaniii as stated by Lemma~\ref{lm:5} and proceed with $G'$.
    
    \begin{figure}[t]
        \centerline{\includegraphics[scale=.9,page=22]{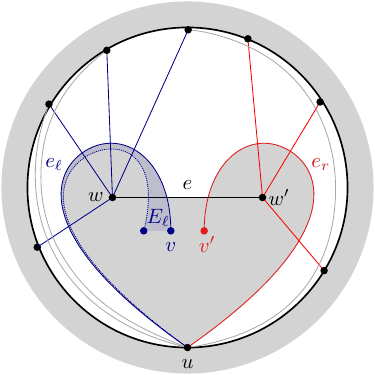}}
        \caption{Illustrations for the proof of Theorem~\ref{thm:density}.}
        \label{fig:theorem2:g}
    \end{figure}
    
    Otherwise, denote by $G_1$ the subgraph of~$G$ consisting of those vertices and edges of~$G$ that lie (entirely) in the \emph{bounded closed} region bounded by~$\mathcal{L}$. In particular, the vertices $u,w,v,v'$, and $w'$, and the edges~$e_\ell$,$e_r$, and $e$ all belong to~$G_1$.
    Similarly, let $G_2$ be the subgraph of~$G$ consisting of those vertices and edges of~$G$ that lie entirely in the \emph{unbounded closed} region bounded by~$\mathcal{L}$.
    In particular, vertex $u \in G_2$, but none of the edges~$e, e_\ell, e_r$ is in~$G_2$. Let $|V[G_2]| = r$ and thus $|V[G_1]| = n-(r-1)$, as $u$ is part of both~$G_1$ and~$G_2$.  

Note that the graph $G$ contains edges  that are neither in $G_1$ nor in $G_2$. 
We will show how to augment $G_2$ to $G_2'$ so that it contains an equal number of extra edges. 
We will create new weakly fan-planar drawings~$\Gamma_1'$ and $\Gamma_2'$ for~$G_1$ and $G_2'$ 
that have at least one fewer edge triple forming Pattern~\reffaniii each.

We start with~$G_1$.
Let $E_\ell$ and $E_r$ be the left valve and the right valve of our heart, respectively, such that $e_\ell \in E_\ell$ and $e_r \in E_r$ holds.
Recall that $G_1$ lies entirely in the bounded region $\mathcal{L}$ defined by the partial curves of $e_\ell$, $e_r$ and $e$ up to their respective intersection points. In particular, this implies that the partial curve of $e_\ell$ between $u$ and $x_\ell$ is crossing free. Moreover, the partial curve of $e$ between $x_\ell$ and $x_r$ is also crossing free as $e_\ell$, $e_r$ and $e$ form a heart. 
Based on this observation, we flip the left valve $E_\ell$, obtaining~$\Gamma_1'$; see Fig.~\ref{fig:theorem2:g1}.

\paragraph{$\Gamma_1'$ is weakly fan-planar and contains fewer Pattern~\reffaniii.}
To show that $\Gamma_1'$ is weakly fan-planar and contains at least one Pattern~\reffaniii less than the drawing of $G_1$ in $\Gamma$, fix an edge $e^\ell_i \in E_\ell$ and let $\gamma^i$ be its curve in $\Gamma_1'$. Recall that 
 by Lemma~\ref{lem:crossing-first-part}, any new crossings of $e^\ell_i$ can only occur on the partial curve $\gamma_1^i \subset \gamma^i$. Since no two edges of $E_\ell$ intersect in $\Gamma_1'$ by construction and any other edge that would intersect $\gamma^i$ would also cross $e_r$ in $\Gamma$ and thus be contained both inside and outside $\mathcal{L}$ by Lemma~\ref{lem:no-two-edges}, $\gamma_1^i$ is uncrossed in $\Gamma_1'$. Hence, all crossings of $e^\ell_i \in E_\ell$ are on the part of $\gamma^i$ which was inherited from $\Gamma$; it follows that no new Pattern \reffani, \reffanii, and \reffaniii is introduced in $\Gamma_1'$.

Hence, our new weakly fan-planar drawing $\Gamma_1'$ has $|E_\ell| \times |E_r| \geq 1$ Pattern~\reffaniii less than $\Gamma$, and we can apply the inductive assumption to get 
	
\[|E(G_1)| \leq 5|V(G_1)| - 10.\]

\begin{figure}[t]
\centering
\begin{subfigure}[b]{.35\textwidth}
\centering
\includegraphics[scale=.9,page=23]{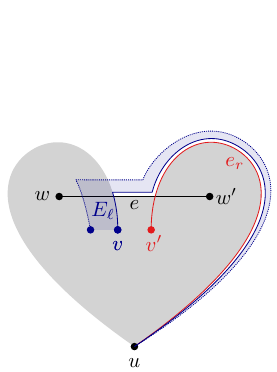}
\subcaption{}
\label{fig:theorem2:g1}
\end{subfigure}
\hfil
\begin{subfigure}[b]{.6\textwidth}
\centering
\includegraphics[scale=.9,page=24]{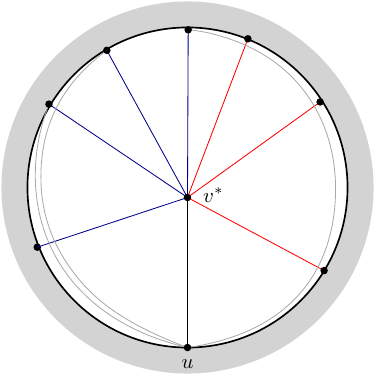}
\subcaption{}
\label{fig:theorem2:g2}
\end{subfigure}
\caption{Illustrations for the proof of Theorem~\ref{thm:density}.}
\label{fig:theorem2}
\end{figure}

\noindent Now we consider $G_2$ and the edges that are neither in $G_1$ nor in $G_2$.

\paragraph{Edges that are neither in $G_1$ nor in $G_2$.} 
Consider such an edge~$e'$. Clearly, $e'$ crosses~$\mathcal{L}$. 
This crossing cannot be on~$e$ by the heart property, so it must be on~$e_\ell$ or~$e_r$. 
The edge~$e'$ cannot cross \emph{both} $e_\ell$ and~$e_r$ by Lemma~\ref{lem:no-two-edges}, so~$e'$ either crosses~$e_\ell$ and must be incident to~$w$ or crosses~$e_r$ and must then be incident to~$w'$ by \reffani and \reffanii.
We claim that $e'$ can only cross edges incident to $u$ outside of $\mathcal{L}$; 
see the light gray edges in Fig.~\ref{fig:theorem2:g}. 
To see this, suppose for a contradiction that $e'$ is crossed by an edge $e''$ outside of $\mathcal{L}$ which is not incident to $u$. 
W.l.o.g. assume that $e'$ is incident to $w$ and crosses $e_\ell$, the other case is symmetric.
By \reffani, $e''$ is then necessarily incident to $v$, which is contained inside $\mathcal{L}$. 
However, we already established that only edges such as $e'$ which are incident to $w$ or $w'$ can leave $\mathcal{L}$, a contradiction to $e''$ crossing $e'$ outside of $\mathcal{L}$.
	 
\paragraph{Augmenting $G_2$.}
Let $k$ be the number of edges of $G$ neither in $G_1$ nor in~$G_2$.
Note that $w$ and $w'$ have at most one common neighbor outside of $\mathcal{L}$ (as otherwise we have already applied Lemma~\ref{lm:5}).
Then, we construct a new graph~$G'_2$ from~$G_2$ by adding edges as follows.
Recall that the weakly fan-planar drawing $\Gamma[G_2]$ derived from $\Gamma$ contains an empty region inside $\mathcal{L}$ and contains fewer triples forming Pattern~\reffaniii, as we removed the heart formed by~$e, e_\ell, e_r$.

We insert a single vertex $v^*$ inside this region, and connect it to all neighbors of $w$ and $w'$ in~$G_2$, and  to $u$; see Fig.~\ref{fig:theorem2:g2}. 
By assumption, $w$ and $w'$ share at most one such vertex and hence there are at least $k-1$ neighbors. Recall that any such edge crosses only edges incident to $u$ outside of $\cal L$ and therefore fan-planarity is maintained. Hence, we augmented $G_2$ to the weakly fan-planar graph $G_2'$ that contains $r+1$ vertices and onto which we can apply the induction hypothesis.
\medskip
 
\noindent We can now bound E(G) as follows:
\begin{align*}
    |E(G)| & = |E(G_1)| + |E(G_2)| + k = |E(G_1)| + |E(G'_2)| \\
    & \leq 5(n-r+1)-10 + 5(r+1)-10 
     = 5n-10
\end{align*}
which concludes the proof.\qed\end{proof}

\begin{figure}[t]
\centering
\begin{subfigure}[b]{.48\textwidth}
\centering
\includegraphics[scale=.9,page=20]{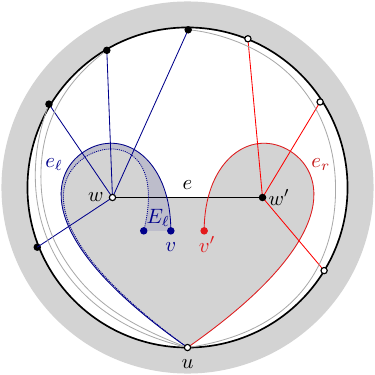}
\subcaption{}
\label{fig:theorem3:g}
\end{subfigure}
\hfil
\begin{subfigure}[b]{.48\textwidth}
\centering
\includegraphics[scale=.9,page=21]{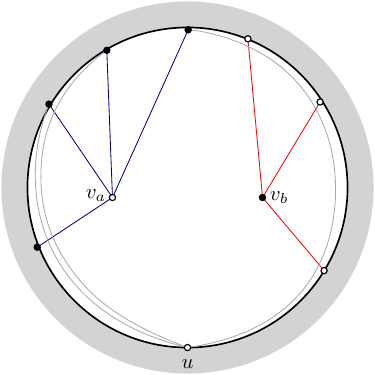}
\subcaption{}
\label{fig:theorem3:g2}
\end{subfigure}
\caption{Illustrations for the proof of Theorem~\ref{thm:density-bipartite}.}
\label{fig:theorem3}
\end{figure}

\noindent For bipartite graphs, we proceed in a similar way.
\begin{theorem} \label{thm:density-bipartite}
An $n$-vertex bipartite weakly fan-planar graph has at most $4n-12$ edges.
\end{theorem}
\begin{proof}
	Let $\Gamma$ be a minimal weakly fan-planar drawing of $G$.
    We again proceed by induction on the number of edge triples forming Pattern~\reffaniii.
    In the base case, $\Gamma$ is strongly fan-planar and
	hence $G$ has at most $4n-12$ edges by~\cite{DBLP:conf/isaac/AngeliniB0PU18}.  
    In the induction step, we proceed as in the previous proof with two differences:
    First, since~$G$ is bipartite, $w$~and~$w'$ cannot have a common neighbor, 
    so we do not need to apply Lemma~\ref{lm:5}.
    Second, when constructing $G_2'$, instead of inserting a single vertex $v^*$, 
    we insert two vertices~$v_a$ and~$v_b$, 
    and connect~$v_a$ with the neighbors of~$w$ in~$G_2$,
    $v_b$~with the neighbors of~$w'$ in~$G_2$; thus maintaining bipartiteness; see Fig.~\ref{fig:theorem3:g2}. 
    In total, we get
	\begin{align*}
	|E(G)| & = |E(G_1)| + |E(G_2)| + k =
    |E(G_1)| + |E(G'_2)| \\
	& \leq 4(n-r+1)-12 + 4(r+2)-12 = 4n-12,
	\end{align*}
	which concludes the proof.
\qed\end{proof}

\bibliographystyle{splncs04}
\bibliography{references}

\ifarxiv
\newpage

\appendix
\section{Proof of Lemma~\ref{lm:5}}
\label{app:5}
\five*

\begin{proof} To prove the claim, will remove $e$, which in turn removes all triples of Pattern~\reffaniii of $\mathcal{H}$, and add another edge $(y_\mathcal{K},y_\mathcal{L}) \notin E$ to account for the missing edge, maintaining weak fan-planarity in the resulting drawing $\Gamma'$ of $G'$. 

Assume w.l.o.g. that $\cal H$ is chosen so that $\cal L$ does not contain any other heart~$\cal H'$.  We aim to find a sequence of curve-segments in $\Gamma$, along which we can insert the edge $(y_\mathcal{K},y_\mathcal{L})$ without introducing Pattern~\reffani to \reffaniii. We refer to the curve constructed in this way as $\gamma$.
Since the edges of $w$ and $w'$ to $p_1$ and $p_2$ both cross edges incident to $u$ they cannot cross each other. Hence there exists a nesting of $p_1$ and $p_2$, w.l.o.g. assume that $p_1$ is the inner one, i.e., the edge $(w,p_1)$ is encountered before the edge $(w,p_2)$ when following $e_\ell$ starting at $v$. In other words, $p_1$ is contained inside a region which we denote as $\mathcal{K}$ that is delimited by $e$,$e_\ell$,$e_r$,$(w,p_2)$ and $(w',p_2)$. Further, denote by $w_p$ and $w_s$ the \emph{predecessor} and \emph{successor} of the edge $e$, i.e., the edge $e_p = (w,w_p)$ crosses $e_\ell$ immediately before $e$ and $e_s = (w,w_s)$ immediately after $e$ when following $e_\ell$ starting at $u$. Note that, possibly $w_p = p_1$ and $w_s=\texttt{NIL}$; see Fig.~\ref{fig:5:1}.

\begin{figure}[t]
\centerline{\begin{subfigure}[b]{.45\textwidth}
\centerline{\includegraphics[scale=1,page=12]{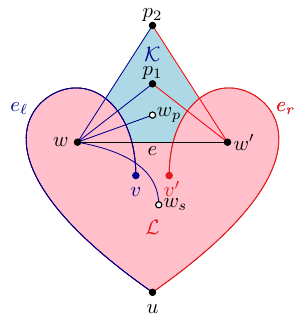}}
\subcaption{}
\label{fig:5:1}
\end{subfigure}
\hfil
\begin{subfigure}[b]{.45\textwidth}
\centerline{\includegraphics[scale=1,page=19]{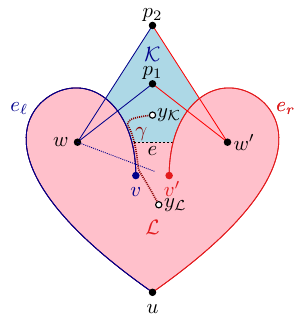}}
\subcaption{}
\label{fig:5:0}
\end{subfigure}}
\caption{(a)~Regions $\mathcal{K}$ and 
$\mathcal{L}$ and vertices $w_p$ and $w_s$. (b)~Replacement of $e$.}
\label{fig:5.0}
\end{figure}

We first identify a suitable vertex $y_\mathcal{K}$ and crossing-free subcurve of $\gamma$ that connects $y_\mathcal{K}$ with the segment of $e$ between $x_\ell$ and $x_r$. 
Consider $w_p$. By Lemma~\ref{lem:no-two-edges}, $(w,w_p)$ cannot cross both $e_\ell$ and $e_r$, hence $w_p$ lies outside of $\mathcal{L}$. Further, since edges of $w$ and $w'$ do not intersect, $w_p$ is contained inside $\mathcal{K}$. Let $x_p$ be the crossing point of $e_p$ and $e_\ell$ in $\Gamma$. Observe that as $e_p$ is the predecessor of $e$, the curve-segment of $e_\ell$ between $x_\ell$ and $x_p$ is crossing-free, hence we add it to $\gamma$. To identify the endpoint $y$, consider the curve-segment $\gamma_p$ of $e_p$ between $x_p$ and $w_p$.

\begin{figure}[t]
\centering
\begin{subfigure}[b]{.45\textwidth}
\centering
\includegraphics[scale=1,page=13]{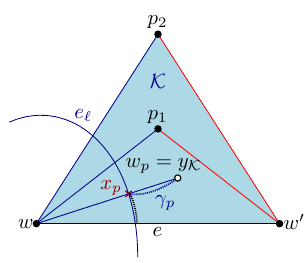}
\subcaption{}
\label{fig:5:2}
\end{subfigure}
\hfil
\begin{subfigure}[b]{.45\textwidth}
\centering
\includegraphics[scale=1,page=14]{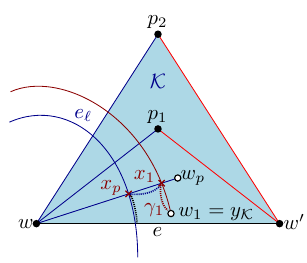}
\subcaption{}
\label{fig:5:3}
\end{subfigure}
\caption{Identification of $y_\mathcal{K}$.}
\label{fig:5.1}
\end{figure}

First, assume that  $\gamma_p$ is crossing-free. In this case, we identify $y_\mathcal{K}$ with $w_p$ and append $\gamma_p$ to $\gamma$; see Fig.~\ref{fig:5:2}.

Second, assume that $\gamma_p$ has at least one crossing. Let $e_1$ be the first edge crossing $e_p$ when traversing $\gamma_p$ from $x_p$. Let $x_1$ be the intersection of $e_1$  and $e_p$. Observe that the segment from $x_1$ to $x_p$ on $e_p$ is crossing-free and hence can be appended to $\gamma$. Since $(x_\ell,x_r)$ is crossing-free by definition, the anchor of $e_p$ must be $u$. Hence, $e_1$ also intersects the edge $(w,p_2)$, its anchor is $w$, and its endpoint $w_1$ lies inside $\mathcal{K}$ (its other endpoint is $u$). Finally, we observe that the curve $\gamma_1$ of $e_1$ between $x_1$ and $w_1$ is crossing free, since otherwise $(w,w_p)$ is not the predecessor of $e$ at $w$. Thus, we fix $w_1$ as $y_\mathcal{K}$ and append $\gamma_1$ to $\gamma$; see Fig.~\ref{fig:5:3}.

Observe that in both cases $y_\mathcal{K}$ lies in $\mathcal{K}$. Also note that $y_\mathcal{K}$ can only be adjacent to $w$ and $w'$ inside $\mathcal{L}$ as the segment of $e$ between $x_\ell$ and $x_r$ is uncrossed whereas $w$ and $w'$ are the anchors of $e_\ell$ and $e_r$, respectively. Hence, we can be sure that $(y_\mathcal{K},y_\mathcal{L})$ does not exist in $G$ yet as long as $y_\mathcal{L} \notin \{w,w',u\}$ and $y_\mathcal{L}$ lies in $\mathcal{L}$; see Fig.~\ref{fig:5:0}. 
Thus, it remains to extend $\gamma$ starting from its intersection with $e$ to a suitable vertex $y_\mathcal{L}$ that is contained in $\mathcal{L}$.

\begin{figure}[t]
\centering
\begin{subfigure}[b]{.45\textwidth}
\centering
\includegraphics[scale=1,page=15]{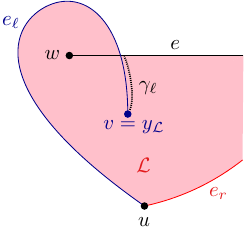}
\subcaption{}
\label{fig:5:4}
\end{subfigure}
\hfil
\begin{subfigure}[b]{.45\textwidth}
\centering
\includegraphics[scale=1,page=16]{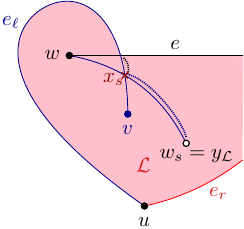}
\subcaption{}
\label{fig:5:5}
\end{subfigure}

\begin{subfigure}[b]{.45\textwidth}
\centering
\includegraphics[scale=1,page=17]{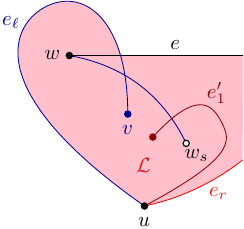}
\subcaption{}
\label{fig:5:6}
\end{subfigure}
\hfil
\begin{subfigure}[b]{.45\textwidth}
\centering
\includegraphics[scale=1,page=18]{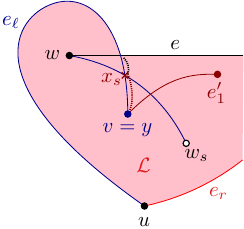}
\subcaption{}
\label{fig:5:7}
\end{subfigure}
\caption{Identification of $y_\mathcal{L}$.}
\label{fig:5.2}
\end{figure}

Consider the curve $\gamma_\ell$ of $e_\ell$ from $x_\ell$ to $v$.  First, if $\gamma_\ell$ is crossing-free, we identify $y_\mathcal{L}$ with $v$ and add $\gamma_\ell$ to $\gamma$; see Fig.~\ref{fig:5:4}.

Otherwise, $\gamma_\ell$ contains a crossing. Assume momentarily that $\gamma_\ell$ is crossed by an edge incident to $w'$. In this scenario, we repeat the argumentation for $w'$ for which this case cannot occur at the same time. 
Thus, we conclude that $e_s$ must be present. Denote by $x_s$ the crossing between $e_\ell$ and $e_s$. We add the curve segment between $x_\ell$ and $x_s$ to $\gamma$.

If the curve of $e_s$ between $w_s$ and $x_s$ is crossing-free, we add it to $\gamma$ and identify $w_s$ with $y_\mathcal{L}$; see Fig.~\ref{fig:5:5}. 

Note that in all cases discussed so far, $\gamma$ is crossing-free when removing $e$. Thus, when replacing $e$ with $(y_\mathcal{K},y_\mathcal{L})$, the number of \reffani stays at most the same while the number of \reffaniii decreases by at least 1; see Fig.~\ref{fig:5:0} (in this case the dotted edge at $w$ does not exist).

Finally, it remains to consider the case where there is an edge $e_1'$  crossing $e_s$. We choose $e_1'$ such that it is the first edge crossing $e_s$ after $x_s$ when traversing $e_s$ from $x_s$ to $w_s$. Note that the anchor of $e_s$ is either $v$ or $u$. If it is $u$, we observe that $e_\ell$, $e_s$ and $e_1'$ form a heart inside $\cal L$ (see Fig.~\ref{fig:5:6}); a contradiction to the choice of $\cal H$. Thus, the anchor of $e_s$ must be $v$. In this case, we again choose $y_\mathcal{L}=v$ and append the segment between $x_s$ and $v$ to $\gamma$. I this case $\gamma$ will cross $(w,w_s)$. However, the anchor of $(w,w_s)$ is $y_\mathcal{L} = v$ so the number of \reffani is maintained. Moreover, since $\gamma$ intersects $(w,w_s)$ in between $e_\ell$ and $e_1'$ the number of \reffaniii still decreases by at least 1 when replacing $e$ with $(y_\mathcal{K},y_\mathcal{L})$; see Fig.~\ref{fig:5:0}. 
\qed\end{proof}
\fi
\end{document}